\documentclass[11pt,a4paper]{amsart}
\usepackage{amssymb,amsmath}
\usepackage{latexsym}
\usepackage{amsthm,amsfonts,amssymb,mathrsfs}
\usepackage{rotating}
\usepackage[leqno]{amsmath}
\usepackage{xspace}
\usepackage[all]{xy}
\usepackage{longtable}
\usepackage{amstext}
\usepackage{amscd}
\usepackage{enumerate}
\usepackage{setspace}
\usepackage{mathrsfs}
\usepackage{verbatim}
\usepackage{tikz}
\usetikzlibrary{matrix,calc}
\headsep=1cm \numberwithin{equation}{section}
\newtheorem{theorem}{Theorem}[section]
\newtheorem{lemma}[theorem]{Lemma}

\newtheorem{corollary}[theorem]{Corollary}

\theoremstyle{definition}
\newtheorem{definition}[theorem]{Definition}
\theoremstyle{remark}

\newtheorem{example}[theorem]{Example}

\newcommand{\Spec}{\operatorname{Spec}}

\newcommand{\id}{\operatorname{id}}

\newcommand{\pd}{\operatorname{pd}}

\newcommand{\Ext}{\operatorname{Ext}}
\newcommand{\Supp}{\operatorname{Supp}}

\newcommand{\Tor}{\operatorname{Tor}}
\newcommand{\Hom}{\operatorname{Hom}}

\newcommand{\lo}{\longrightarrow}
\newcommand{\fm}{\mathfrak{m}}

\newenvironment{prf}[1][Proof]{\begin{proof}[\bf #1]}{\end{proof}}

\DeclareMathSymbol{\perp}{\mathrel}{symbols}{"3F}

\begin{document}

\author[K. Abolfath Beigi, K. Divaani-Aazar and M. Tousi]{Kosar Abolfath Beigi, Kamran Divaani-Aazar and
Massoud Tousi}

\title[The balance of relative Ext ...]
{The balance of relative Ext groups defined by semi dualizing modules}

\address{K. Abolfath Beigi, Department of Mathematics, Faculty of Mathematical Sciences, Alzahra
University, Tehran, Iran.}
\email{kosarabolfath@gmail.com}

\address{K. Divaani-Aazar, Department of Mathematics, Faculty of Mathematical Sciences, Alzahra
University, Tehran, Iran-and-School of Mathematics, Institute for Research in Fundamental Sciences (IPM),
P.O. Box 19395-5746, Tehran, Iran.}
\email{kdivaani@ipm.ir}

\address{M. Tousi, Department of Mathematics, Faculty of Mathematical Sciences, Shahid Beheshti University,
Tehran, Iran.}
\email{mtousi@ipm.ir}



\begin{abstract} Let $R$ be a commutative Noetherian ring with identity and $C$ a semidualizing module for $R$.
Let $\mathscr{P}_C(R)$ and $\mathscr{I}_C (R)$ denote, respectively, the classes of $C$-projective and $C$-injective
$R$-modules. We show that their induced Ext bifunctors $\Ext^i_{\mathscr{P}_C}(-,\sim)$ and $\Ext^i_{\mathscr{I}_C}
(-,\sim)$ coincide for all $i\geq 0$ if and only if $C$ is projective. Also, we provide some other criteria for $C$
to be projective by using some special cotorsion theories.
\end{abstract}

\maketitle

\section{Introduction}

Throughout this paper, $R$ is a commutative Noetherian ring with nonzero identity.  Foxby \cite{3}, Vasconcelos
\cite{9} and Golod \cite{4}, independently, initiated the study of semidualizing modules under other names. The
notion of semidualizing modules is a generalization of that of dualizing modules, in fact, a dualizing module is
a semidualizing module with finite injective dimension.

The aim of this paper is to establish some characterizations of projective semidualizing modules. (It is known
that every projective semidualizing module has rank 1.) To achieve this goal, we apply two approaches. In Section
3, we use relative homologies defined by semidualizing modules, and in Section 4, we use special cotorsion theories.

Let $C$ be a semidualizing module for $R$. Enochs, Holm, White and Takahashi investigated the classes of $C$-projective,
$C$-flat and $C$-injective $R$-modules, denoted, respectively, by $\mathscr{P}_C(R)$, $\mathscr{F}_C(R)$ and
$\mathscr{I}_C(R)$; see \cite{2},\cite{5} and \cite{8}. The classes $\mathscr{P}_C(R)$ and $\mathscr{I}_C(R)$
can be used for defining the relative Ext bifunctors $\Ext^i_{\mathscr{P}_C} (-,\sim)$ and $\Ext^i_{\mathscr{I}_C}(-,\sim)$,
which are extensively examined by several authors; see e. g.  \cite{7} and \cite{8}. In particular, Sather-Wagstaff,
Sharif and White \cite{7} proved that $\Ext^i_{\mathscr{P}_C} (M,N)$ is not isomorphic to $\Ext^i_{\mathscr{I}_C}(M,N)$
in general. We show that $\Ext^i_{\mathscr{P}_C}(M,N)\cong \Ext^i_{\mathscr{I}_C}(M,N)$ for all $R$-modules $M$ and $N$
and all $i\geq 0$ if and only if $C$ is projective; see Corollary \ref{3.6}.

There is a rich theory of cotorsion theories. They are venerable tools for studying precovering and preenveloping classes
of $R$-modules. We show that $C$ is projective if and only if any of the pairs $(\mathscr{P}_C (R),\mathscr{P}_C (R)^{\perp})$,
$(\mathscr{F}_C (R),\mathscr{F}_C(R)^{\perp})$ or $(^{\perp}\mathscr{I}_C (R),\mathscr{I}_C (R))$ is a cotorsion theory; see
Theorem \ref{4.3}.

\section{Prerequisites}

We start with recalling the following three definitions.

\begin{definition} A finitely generated $R$-module $C$ is called {\it semidualizing}, if it satisfies the
following conditions:
\begin{enumerate}
\item[(i)] The homothety map $\chi_C^R:R \lo \Hom_R(C,C)$ is an isomorphism, and
\item[(ii)] $\Ext^i_R(C,C)=0$ for all $i>0$.
\end{enumerate}
\end{definition}

\begin{definition} Let $C$ be a  semidualizing module for $R$.
\begin{enumerate}
\item[(i)] The {\it Auslander class} $\mathscr{A}_C(R)$ is the class of all $R$-modules $M$ for which the natural map
$$\gamma_M^C:M\lo \Hom_R(C,C\otimes_RM)$$ is an isomorphism, and $\Tor^R_i(C,M)=0=\Ext_R^i(C,C\otimes_RM)$ for all $i>0$.
\item[(ii)] The {\it Bass class} $\mathscr{B}_C(R)$ is the class of all $R$-modules $M$ for which the evaluation map
$$\xi_M^C:C\otimes_R\Hom_R(C,M)\lo M$$ is an isomorphism, and $\Ext^i_R(C,M)=0=\Tor^R_i(C,\Hom_R(C,M))$ for all $i>0$.
\end{enumerate}
\end{definition}

Let $\mathscr{P}(R)$, $\mathscr{F}(R)$ and $\mathscr{I}(R)$ denote, respectively, the full subcategories of projective, flat
and injective $R$-modules. Clearly, $\mathscr{F}(R)\subseteq \mathscr{A}_C(R)$ and $\mathscr{I}(R)\subseteq \mathscr{B}_C(R)$.
Also, by Foxby equivalence, there is an equivalence of categories:

\begin{displaymath}
\xymatrix{\mathscr{A}_C\left(R\right) \ar@<0.7ex>[rrr]^-{C\otimes_R-} &
{} & {} & \mathscr{B}_C\left(R\right).  \ar@<0.7ex>[lll]^-{\Hom_R\left(C,-\right)}}
\end{displaymath}

\begin{definition} Let $C$ be a  semidualizing module for $R$. The classes of $C$-{\it projective},
$C$-{\it flat} and $C$-{\it injective} $R$-modules are defined by
$$\mathscr{P}_C(R):=\{ C\otimes_R P \mid P\in \mathscr{P}(R)\},$$
$$\mathscr{F}_C(R):=\{ C\otimes_R F \mid F\in \mathscr{F}(R)\},$$
$$\mathscr{I}_C(R):=\{ \Hom_R(C,I) \mid I\in \mathscr{I}(R)\}.$$
\end{definition}

Every $R$-module $M$ admits an {\it augmented proper $\mathscr{P}_C$-projective resolution}. That is a complex
$$X^+=\cdots \longrightarrow C\otimes_R P_n \longrightarrow \cdots \longrightarrow C\otimes_R P_0 \longrightarrow M
\longrightarrow 0$$ such that every $P_i$ is projective, and the complex
$$\Hom_R(C,X^+)=\cdots \longrightarrow P_n \longrightarrow \cdots \longrightarrow  P_0 \longrightarrow \Hom_R(C,M)
\longrightarrow 0$$ is exact. The truncated complex $$X=\cdots \longrightarrow C\otimes_R P_n \longrightarrow \cdots
\longrightarrow C\otimes_R P_0 \longrightarrow 0$$ is called a {\it proper $\mathscr{P}_C$-projective resolution}
of $M$. The {\it $\mathscr{P}_C$-projective dimension} of $M$ is defined by
$$\mathscr{P}_C\text{-}\pd_R(M):=\inf\{\sup\{n|X_n\neq 0\} \mid~ X~\text{is}~\text{a} ~\text{proper}~
\mathscr{P}_C\text{-}\text{projective}~\text{resolution}~\text{of}~M\}.$$

Every $R$-module $M$ admits an {\it augmented proper $\mathscr{I}_C$-injective resolution}. That is a complex
$$Y^+=0 \longrightarrow M \longrightarrow \Hom_R(C,I^0) \longrightarrow \Hom_R(C,I^1) \longrightarrow \cdots
\longrightarrow \Hom_R(C,I^n)\longrightarrow \cdots$$ such that every $I^i$ is injective, and the complex
$$C\otimes_R Y^+=0\longrightarrow C\otimes_RM\longrightarrow I^0 \longrightarrow I^1\longrightarrow \cdots
\longrightarrow I^n \longrightarrow \cdots$$ is exact. The truncated complex $$Y=0 \longrightarrow
\Hom_R(C,I^0)\longrightarrow\Hom_R(C,I^1)\longrightarrow \cdots \longrightarrow \Hom_R(C,I^n)\longrightarrow \cdots$$
is called a {\it proper $\mathscr{I} _C$-injective resolution} of $M$. The {\it $\mathscr{I} _C$-injective dimension}
of $M$ is defined by $$\mathscr{I}_C\text{-}\id_R(M):= \inf\{\sup\{ n | Y^n\neq0 \} \mid~ Y~\text{is}~\text{a}~\text{proper}
~\mathscr{I}_C\text{-}\text{injective}~\text{resolution}~\text{of}~M\}.$$

Let $M$ and $N$ be two $R$-modules. Assume that $L$ and $J$ are, respectively, a proper $\mathscr{P}_C$-projective
resolution of $M$ and a proper $\mathscr{I}_C$-injective resolution of $N$. For each non-negative integer $n$,
the $n$th relative Ext modules of $M$ and $N$ are defined by $$\Ext^n_{\mathscr{P}_C}(M,N):=\text{H}_{-n}(\Hom_R(L,N)),
\ \text{and}$$
$$\Ext^n_{\mathscr{I}_C}(M,N):=\text{H}_{-n}(\Hom_R(M,J)).$$

The following result reveals some important properties of $\mathscr{P}_C$-projective and $\mathscr{I}_C$-injective dimensions.

\begin{lemma}\label{2.4} (See \cite[Theorems 3.2 and 4.1 and Corollary 2.10]{8}.) Let $C$ be a semidualizing module for
$R$ and $M$ and $N$ two $R$-modules. Then
\begin{itemize}
\item[(i)]  $M$ is a nonzero $C$-projective $R$-module if and only if $\mathscr{P}_C\text{-}\pd_R(M)=0$.
\item[(ii)] $M$ is a nonzero $C$-injective $R$-module if and only if $\mathscr{I}_C\text{-}\id_R(M)=0$.
\item[(iii)] $\mathscr{P}_C\text{-}\pd_R(M)=\sup \{n\in \mathbb{N}_0 \mid \Ext^n_{\mathscr{P}_C}(M,-)\neq 0\}$.
\item[(iv)] $\mathscr{I}_C\text{-}\id_R(M)=\sup \{n\in \mathbb{N}_0 \mid \Ext^n_{\mathscr{I}_C}(-,M)\neq 0\}$.
\item[(v)] $\Ext^n_{\mathscr{P}_C}(M,N)\cong \Ext^n_R(\Hom_R(C,M),\Hom_R(C,N))$ for all $n\geq 0$.
\item[(vi)] $\Ext^n_{\mathscr{I}_C}(M,N)\cong \Ext^n_R(C\otimes_RM,C\otimes_RN)$ for all $n\geq 0$.
\end{itemize}
\end{lemma}

The Picard group of $R$ is the set $\operatorname{Pic}(R)$ of all isomorphism classes of finitely generated projective
$R$-modules of rank $1$. The isomorphism class of a given finitely generated projective $R$-module $P$ of rank $1$ is
denoted by $[P]$. The set $\operatorname{Pic}(R)$ has the structure of an Abelian group with the multiplication
$[P][Q]:=[P\otimes_RQ]$. The identity element of $\operatorname{Pic}(R)$ is $[R]$, and inverses are given by the formula
$[P]^{-1}=[\Hom_R(P,R)]$. According to \cite[Corollary 2.2.5]{6} any finitely generated projective $R$-module of rank 1
is semidualizing. Consequently, every ring with non-trivial Picard group possesses a projective semidualizing
module which is not free.

Let $\mathcal{G}_0(R)$ denote the set of all isomorphism classes of semidualizing $R$-modules. Then $\operatorname{Pic}(R)
\subseteq \mathcal{G}_0(R)$. There is a well-defined action of $\operatorname{Pic}(R)$ on $\mathcal{G}_0(R)$ given by
$[P][C]:=[C\otimes_RP]$. The symbol $\approx$ stands for the equivalence relation on $\mathcal{G}_0(R)$ that induced by
this action. So, for $[B],[C]\in \mathcal{G}_0(R)$, we have $[B]\approx [C]$, if there is an element $P\in \operatorname{Pic}(R)$
such that $C\cong B\otimes_RP$. If $B$ and $C$ are two semidualizing modules for $R$ such that $C\cong B\otimes_RF$ for some
flat $R$-module $F$, then the equivalence (i) $\Leftrightarrow$ (ii) in \cite[Proposition 4.1.4]{6} yields that $[B]\approx [C]$.

\section{Balancedness}

We begin with the following useful result. Recall that an $R$-module $I$ is called {\it faithfully injective} if the functor
$\Hom_R(-,I)$ is faithfully exact. A faithfully injective module is obviously injective. Any ring possesses a faithfully
injective module; see e.g. \cite[Definition A.2.2 and Example A.2.3]{6}.

\begin{lemma}\label{3.1} Let $C$ be a semidualizing module for $R$. The following are equivalent:
\begin{itemize}
\item[(i)] $C$ is projective.
\item[(ii)] $\pd_RC< \infty$.
\item[(iii)]  $R$ is $C$-projective.
\item[(iv)]  $\mathscr{P}_C\text{-}\pd_R(R)<\infty$.
\item[(v)] $[C]\approx [R]$.
\item[(vi)] There exists a faithfully injective $R$-module $E$ which is $C$-injective.
\item[(vii)] $C_{\fm}\cong R_{\fm}$ for every maximal ideal $\fm$ of $R$.
\end{itemize}
Furthermore, if $C$ is projective, then both of $\mathscr{A}_C(R)$ and $\mathscr{B}_C(R)$ are the class of all
$R$-modules.
\end{lemma}

\begin{proof} The last assertion is trivial. So, we only show the claimed equivalences.

(i)$\Longleftrightarrow$(ii) is obvious. Note that by \cite[Corollary 2.2.8]{6}, every semidualizing module for $R$
with finite projective dimension is projective.

(iii)$\Longleftrightarrow$(v) and (i)$\Longleftrightarrow$(vii) are immediate, respectively by the equivalences
(i) $\Leftrightarrow$ (ii) and (ii) $\Leftrightarrow$ (iv) in \cite[Proposition 4.1.4]{6}.

(i) $\Rightarrow$ (iii) Since $C$ is a finitely generated projective $R$-module, it turns out that $\Hom_R(C,R)$ is
also a finitely generated projective $R$-module. Tensor evaluation implies that $$R\cong \Hom_R(C,C)\cong \Hom_R(C,
R\otimes_RC)\cong C\otimes_R\Hom_R(C,R).$$

(v) $\Rightarrow$ (i) Assume that $R\cong C\otimes_RP$ for some finitely generated projective $R$-module $P$. Then
$$\begin{array}{ll}
C& \cong \Hom_R(R,C)\\
& \cong \Hom_R(C\otimes_RP,C)\\
&\cong \Hom_R(P,\Hom_R(C,C))\\
&\cong \Hom_R(P,R),\\
\end{array}$$
and so $C$ is projective.

(iii) $\Rightarrow$ (iv) is trivial.

(iv) $\Rightarrow$ (i) \cite[Corollary 2.9]{8} yields that $R\in \mathscr{B}_C(R)$, and so $C$ is projective by
\cite[Corollary 4.1.6]{6}.

(i) $\Rightarrow$ (vi) Let $I$ be a faithfully injective $R$-module. Then $\Hom_R(C,I)$ is a $C$-injective $R$-module.
Since $C$ is projective, by adjointness, it follows that the $R$-module $\Hom_R(C,I)$ is injective. Suppose that
$\Hom_R(M,\Hom_R(C,I))=0$ for some $R$-module $M$. Then, adjointness yields that $\Hom_R(C,\Hom_R(M,I))=0$. This
implies that $\Hom_R(M,I)=0$, because $\Supp_RC=\Spec R$. So, $M=0$. Thus the $C$-injective $R$-module $\Hom_R(C,I)$
is faithfully injective.

(vi) $\Rightarrow$ (i) There is an injective $R$-module $I$ such that $E\cong \Hom_R(C,I)$. We show that $I$ is
faithfully injective. Suppose that $\Hom_R(M,I)=0$ for some $R$-module $M$. Then,
$$\begin{array}{ll}
\Hom_R(M,E)& \cong \Hom_R(M,\Hom_R(C,I))\\
&\cong\Hom_R(C,\Hom_R(M,I))\\
&= 0.\\
\end{array}$$
Since $E$ is faithfully injective, it follows that $M=0$. Thus $I$ is faithfully injective. As
$$\Hom_R(-\otimes_RC,I)\cong \Hom_R(-,\Hom_R(C,I)),$$ it follows that $\Hom_R(-\otimes_RC, I)$ is an exact functor.
Since $I$ is faithfully injective, the functor $-\otimes_R C$ is exact too. So, $C$ is flat. As $C$ is finitely
generated, it follows that $C$ is projective.

The last assertion of the lemma follows by \cite[Corollary 4.1.6]{6}.
\end{proof}

In the proofs of Theorems \ref{3.3} and \ref{3.4}, we apply the following useful result.

\begin{lemma}\label{3.2} Let $B$ and $C$ be two semidualizing modules for $R$.  The following are equivalent:
\begin{itemize}
\item[(i)] $[B]\approx [C]$.
\item[(ii)] $\mathscr{P}_B(R)=\mathscr{P}_C(R)$.
\item[(iii)] $\mathscr{F}_B(R)=\mathscr{F}_C(R)$.
\item[(iv)] $\mathscr{I}_B(R)=\mathscr{I}_C(R)$.
\item[(v)] For any given $R$-module $M$, a complex $X$ is an augmented proper $\mathscr{P}_B$-projective resolution of
$M$ if and only if it is an augmented proper $\mathscr{P}_C$-projective resolution of $M$.
\item[(vi)] For any given $R$-module $M$, a complex $X$ is an augmented proper $\mathscr{I}_B$-injective resolution of
$M$ if and only if it is an augmented proper $\mathscr{I}_C$-injective resolution of $M$.
\item[(vii)]  $\mathscr{P}_B\text{-}\pd_R(M)=\mathscr{P}_C\text{-}\pd_R(M)$ for every $R$-module $M$.
\item[(viii)]  $\mathscr{I}_B\text{-}\id_R(M)=\mathscr{I}_C\text{-}\id_R(M)$ for every $R$-module $M$.
\end{itemize}
\end{lemma}

\begin{prf} (i) $\Rightarrow$ (ii) and (i) $\Rightarrow$ (iii) There are projective $R$-modules $P$ and $Q$ such that
$B\cong C\otimes_RP$ and $C\cong B\otimes_RQ$. Since the tensor product of every two projective (flat) $R$-modules is
projective (flat), it follows that $\mathscr{P}_B(R)=\mathscr{P}_C(R)$ and $\mathscr{F}_B(R)=\mathscr{F}_C(R)$.

(ii) $\Rightarrow$ (i) Since $C\in \mathscr{P}_C(R)$, the equality $\mathscr{P}_B(R)=\mathscr{P}_C(R)$, yields the
existence of a projective $R$-module $P$ such that $C\cong B\otimes_RP$. Now, the claim follows by the equivalence (i)
$\Leftrightarrow$ (ii) in \cite[Proposition 4.1.4]{6}.

(iii) $\Rightarrow$ (i) follows exactly by the proof of (ii) $\Rightarrow$ (i).

(i) $\Rightarrow$ (iv) There are projective $R$-modules $P$ and $Q$ such that $B\cong C\otimes_RP$ and $C\cong
B\otimes_RQ$. Let $I$ be an injective $R$-module. Adjointness implies that $$\Hom_R(B,I)\cong \Hom_R(C,\Hom_R(P,I))$$
and $$\Hom_R(C,I)\cong \Hom_R(B,\Hom_R(Q,I)).$$  As the $R$-modules $\Hom_R(P,I)$ and $\Hom_R(Q,I)$ are injective, it
follows that $\mathscr{I}_B(R)=\mathscr{I}_C(R)$.

(iv) $\Rightarrow$ (i) Let $E$ be a faithfully injective $R$-module. Then $\Hom_R(C,E)$ is a $C$-injective $R$-module.
As $\mathscr{I}_C(R)\subseteq \mathscr{I}_B(R)$, there is an injective $R$-module $I$ such that $\Hom_R(C,E)\cong \Hom_R(B,I)$.
Thus, by Foxby equivalence, $\Hom_R(C,E)\in \mathscr{A}_B(R)$. Now, \cite[Proposition 3.3.1]{6} implies that $C\in
\mathscr{B}_B(R)$. In the same way, the inclusion $\mathscr{I}_B(R) \subseteq \mathscr{I}_C(R)$ implies that $B\in
\mathscr{B}_C(R)$. Now, the claim follows by the equivalence (i) $\Leftrightarrow$ (vii) in \cite[Proposition 4.1.4]{6}.

(ii) $\Rightarrow$ (v) Let $M$ be an $R$-module. By the symmetry, it is enough to show that every augmented proper
$\mathscr{P}_B$-projective resolution of $M$ is an augmented proper $\mathscr{P}_C$-projective resolution of $M$.
Let $$X^+=\cdots \rightarrow B\otimes_RP_2 \rightarrow B\otimes_RP_1\rightarrow B\otimes_RP_0 \rightarrow M\rightarrow 0$$ be
an augmented proper $\mathscr{P}_B$-projective resolution of $M$. As every $B$-projective $R$-module is $C$-projective, it is
enough to show that $\Hom_R(C,X^+)$ is exact. Since $C\in \mathscr{P}_C(R)=\mathscr{P}_B(R)$, it follows that $C\cong B\otimes_RP$
for some projective $R$-module $P$. Now, $$\Hom_R(C,X^+)\cong\Hom_R(B\otimes_RP,X^+)\cong \Hom_R(P,\Hom_R(B,X^+)),$$
and so $\Hom_R(C,X^+)$ is exact.

(iv) $\Rightarrow$ (vi) Let $M$ be an $R$-module. Let $$Y^+= 0\rightarrow M \rightarrow \Hom_R(B,I^0)\rightarrow \Hom_R(B,I^1)
\rightarrow \cdots $$ be an augmented proper $\mathscr{I}_B$-injective resolution of $M$. As every $B$-injective $R$-module is
$C$-injective, it is enough to show that $C\otimes _R Y^+$ is exact. We have shown above that (iv) and (i) are equivalent. So,
$C\cong B\otimes_RP$ for some projective $R$-module $P$. Now, $$C\otimes_R Y^+ \cong (B\otimes_RP)\otimes_RY^+\cong P\otimes_R
(B\otimes_RY^+),$$ and so $C\otimes_RY^+$ is exact. The reverse follows by symmetry.

(v) $\Rightarrow$ (vii) and (vi) $\Rightarrow$ (viii) are trivial by the definition.

(vii) $\Rightarrow$ (ii) and (viii) $\Rightarrow$ (iv) follow immediately by parts (i) and (ii) of Lemma \ref{2.4}.
\end{prf}

The next result is the first main result of this paper.

\begin{theorem}\label{3.3} Let $B$ and $C$ be two semidualizing modules for $R$. The following are equivalent:
\begin{itemize}
\item[(i)] $[B]\approx [C]$.
\item[(ii)] $\Ext^i_{\mathscr{P}_C}(M,N)\cong \Ext^i_{\mathscr{P}_B}(M,N)$ for all $R$-modules $M$ and $N$ and $i\geq 0$.
\item[(iii)] $\Ext^i_{\mathscr{I}_C}(M,N)\cong \Ext^i_{\mathscr{I}_B}(M,N)$ for all $R$-modules $M$ and $N$ and $i\geq 0$.
\end{itemize}
\end{theorem}

\begin{proof} (i) $\Rightarrow$ (ii) and (i) $\Rightarrow$ (iii) are clear by Lemma \ref{3.2}.

(ii) $\Rightarrow$ (i) In view of Lemma \ref{2.4}(iii), we deduce that $\mathscr{P}_B\text{-}\pd_R(M)=\mathscr{P}_C\text{-}\pd_R(M)$
for every $R$-module $M$.  Hence, (i) holds by Lemma \ref{3.2}.

(iii) $\Rightarrow$ (i) In view of Lemma \ref{2.4}(iv), we deduce that $\mathscr{I}_B\text{-}\id_R(M)=\mathscr{I}_C\text{-}\id_R(M)$
for every $R$-module $M$.  Hence, (i) holds by Lemma \ref{3.2}.
\end{proof}

In what follows, for a finitely generated module $M$ over a local ring $(R,\fm,k)$, let $\mu_R(M)$ denote the minimum number
of generators of $M$. Also, let $\text{Vdim}_kL$ denote the dimension of a $k$-vector space $L$.  It is known that $\mu_R(M)=\text{Vdim}_k(k\otimes_RM)$. In particular, if $N$ is another finitely generated module over $(R,\fm,k)$, then
$$\begin{array}{ll}
\mu_R(M\otimes_RN)&=\text{Vdim}_k(k\otimes_R(M\otimes_RN))\\
&=\text{Vdim}_k((k\otimes_RM)\otimes_k(k\otimes_RN))\\
&=\text{Vdim}_k(k\otimes_RM)\times \text{Vdim}_k(k\otimes_RN)\\
&=\mu_R(M)\mu_R(N).
\end{array}
$$
For every finitely generated $R$-module $M$ and maximal ideal $\fm$ of $R$, one may check that $\mu_{R_{\fm}}(M_{\fm})=
\text{Vdim}_{R/\fm}(R/\fm\otimes_RM)$.

\begin{theorem}\label{3.4} Let $B$ and $C$ be two semidualizing modules for $R$. The following are equivalent:
\begin{itemize}
\item[(i)] $[B]\approx [R]\approx [C]$.
\item[(ii)] $\Ext^i_{\mathscr{P}_C}(M,N)\cong \Ext^i_{\mathscr{I}_B}(M,N)$ for all $R$-modules $M$ and $N$ and $i\geq 0$.
\item[(iii)] $\Ext^i_{\mathscr{P}_C}(C,R/\fm)\cong \Ext^i_{\mathscr{I}_B}(C,R/\fm)$ for all maximal ideals $\fm$ of
$R$ and $i\geq 0$.
\item[(iv)] $\Ext^0_{\mathscr{P}_C}(C,R/\fm)\cong \Ext^0_{\mathscr{I}_B}(C,R/\fm)$ and $\Ext^0_{\mathscr{P}_C}(B,R/\fm)
\cong \Ext^0_{\mathscr{I}_B}(B,R/\fm)$ for all maximal ideals $\fm$ of $R$.
\end{itemize}
\end{theorem}

\begin{proof} (i) $\Rightarrow$ (ii) follows by Theorem \ref{3.3}.
	
(ii) $\Rightarrow$ (iii) and (ii) $\Rightarrow$ (iv) are obvious.

(iv) $\Rightarrow$ (i) By Lemma \ref{3.1}, it is enough to show that $B_{\fm}\cong R_{\fm}\cong C_{\fm}$ for every maximal ideal
$\fm$ of $R$. Let $\fm$ be a maximal ideal of $R$, and set $k:=R/\fm$. By parts (v) and (vi) of Lemma \ref{2.4}, we have:
$$\begin{array}{ll}
\Ext^0_{\mathscr{P}_C}(C,k)&\cong \Ext^0_R(\Hom_R(C,C),\Hom_R(C,k))\\
&\cong \Hom_R(C,k)\\
&\cong \Hom_R(C,\Hom_k(k,k))\\
&\cong \Hom_k(C\otimes_Rk,k)\\
&\cong k^{\mu_{R_{\fm}}(C_{\fm})},
\end{array}
$$
and
$$\begin{array}{ll}
\Ext^0_{\mathscr{I}_B}(C,k)& \cong \Ext^0_R(B\otimes_RC,B\otimes_Rk)\\
&\cong \Hom_R(B\otimes_RC,\Hom_k(k,B\otimes_Rk))\\
&\cong \Hom_k((B\otimes_RC)\otimes_Rk,B\otimes_Rk)\\
&\cong k^{\mu_{R_{\fm}}(C_{\fm})\mu_{R_{\fm}}(B_{\fm})^2}.
\end{array}
$$
Hence, $\mu_{R_{\fm}}(C_{\fm})=\mu_{R_{\fm}}(C_{\fm})\mu_{R_{\fm}}(B_{\fm})^2$, and so $\mu_{R_{\fm}}(B_{\fm})=1$.
Thus $B_{\fm}$ is a cyclic semidualizing module for the ring $R_{\fm}$. Now,  \cite[Corollary 2.1.14]{6} implies
that $B_{\fm}\cong R_{\fm}$.

Set $t:=\mu_{R_{\fm}}(\Hom_R(C,B)_{\fm})$. Then, we have:
$$\begin{array}{ll}
\Ext^0_{\mathscr{P}_C}(B,k)&\cong \Ext^0_R(\Hom_R(C,B),\Hom_R(C,k))\\
&\cong \Hom_R(\Hom_R(C,B)\otimes_RC,k)\\
&\cong \Hom_R(\Hom_R(C,B)\otimes_RC,\Hom_k(k,k))\\
&\cong \Hom_k((\Hom_R(C,B)\otimes_RC)\otimes_Rk,k)\\
&\cong k^{t\mu_{R_{\fm}}(C_{\fm})},
\end{array}
$$
and
$$\begin{array}{ll}
\Ext^0_{\mathscr{I}_B}(B,k)& \cong \Ext^0_R(B\otimes_RB,B\otimes_Rk)\\
&\cong \Hom_R(B\otimes_RB,\Hom_k(k,B\otimes_Rk))\\
&\cong \Hom_k((B\otimes_RB)\otimes_Rk,B\otimes_Rk)\\
&\cong k^{\mu_{R_{\fm}}(B_{\fm})^3}\\
&\cong k.
\end{array}
$$
Hence $t\mu_{R_{\fm}}(C_{\fm})=1$, and so $\mu_{R_{\fm}}(C_{\fm})=1$. So, again applying \cite[Corollary 2.1.14]{6} yields that
$C_{\fm}\cong R_{\fm}$.

(iii) $\Rightarrow$ (i) Again by Lemma \ref{3.1}, it is enough to show that $B_{\fm}\cong R_{\fm}\cong C_{\fm}$ for every
maximal ideal $\fm$ of $R$. Let $\fm$ be a maximal ideal of $R$. By the proof of (iv) $\Rightarrow$ (i), we conclude that
$B_{\fm}\cong R_{\fm}$. Clearly, $C$ is $C$-projective, and so by parts (i) and (iii) of Lemma \ref{2.4}, we conclude that $\Ext^i_{\mathscr{P}_C}(C,R/\fm)=0$
for all $i\geq 1$. For every $i\geq 1$, Lemma \ref{2.4}(vi) implies that:
$$\begin{array}{ll}
\Ext^i_{R_{\fm}}(C_{\fm},R_{\fm}/\fm R_{\fm})&\cong \Ext^i_{R_{\fm}}(R_{\fm}\otimes_{R_{\fm}}C_{\fm},R_{\fm}
\otimes_{R_{\fm}}(R_{\fm}/\fm R_{\fm}))\\
& \cong \Ext^i_{R_{\fm}}(B_{\fm}\otimes_{R_{\fm}}C_{\fm},B_{\fm}
\otimes_{R_{\fm}}(R_{\fm}/\fm R_{\fm}))\\
& \cong \Ext^i_R(B\otimes_RC,B\otimes_RR/\fm)_{\fm}\\
& \cong \Ext^i_{\mathscr{I}_B}(C,R/\fm)_{\fm}\\
&\cong \Ext^i_{\mathscr{P}_C} (C,R/\fm)_{\fm}\\
&=0.
\end{array}
$$
Therefore, $C_{\fm}$ is a projective semidualizing module for the local ring $R_{\fm}$, and so $C_{\fm}\cong R_{\fm}$.
\end{proof}

Let $B$ and $C$ be two semidualizing modules for a local ring $(R,\fm,k)$. In view of the fourth part of the above
result, it might be natural to ask: Does $\Ext^0_{\mathscr{P}_C}(C,k)\cong \Ext^0_{\mathscr{I}_B}(C,k)$ imply that
$B\cong R \cong C$?  The answer is negative, in view of the following example.

\begin{example}\label{3.5} Let $(R,\fm,k)$ be a local ring and $C$ a semidualizing module for $R$ such that $C\ncong R$.
Lemma \ref{2.4}(v) implies that $\Ext^0_{\mathscr{P}_C}(C,k)\cong \Ext^0_R(C,k)$, while $C\ncong R$.
\end{example}

We record the following corollary. It shows that for a semidualizing module $C$ for $R$, the two relative Ext modules
defined by $C$ coincide for all $i\geq 0$ if and only if $C$ is projective.

\begin{corollary}\label{3.6} Let $C$ be a  semidualizing module for $R$. The following conditions are equivalent:
\begin{itemize}
\item[(i)] $[C]\approx [R]$.
\item[(ii)] $\Ext^i_{\mathscr{P}_C}(M,N)\cong \Ext^i_{\mathscr{I}_C}(M,N)$ for all $R$-modules $M$ and $N$ and $i\geq 0$.
\item[(iii)] $\Ext^i_{\mathscr{P}_C}(C,R/\fm)\cong \Ext^i_{\mathscr{I}_C}(C,R/\fm)$ for all maximal ideals $\fm$ of $R$ and
$i\geq 0$.
\item[(iv)] $\Ext^0_{\mathscr{P}_C}(C,R/\fm)\cong \Ext^0_{\mathscr{I}_C}(C,R/\fm)$ for all maximal ideals $\fm$ of $R$.
\item[(v)] $\Ext^i_{\mathscr{P}_C}(M,N)\cong \Ext^i_R(M,N)$ for all $R$-modules $M$ and $N$ and $i\geq 0$.
\item[(vi)] $\Ext^i_{\mathscr{P}_C}(C,R/\fm)\cong \Ext^i_R(C,R/\fm)$ for all maximal ideals $\fm$ of $R$ and $i\geq 0$.
\item[(vii)] $\Ext^i_{\mathscr{I}_C}(M,N)\cong \Ext^i_R(M,N)$ for all $R$-modules $M$ and $N$ and $i \geq 0$.
\item[(viii)] $\Ext^i_{\mathscr{I}_C}(C,R/\fm)\cong \Ext^i_R(C,R/\fm)$ for all maximal ideals $\fm$ of $R$ and $i \geq 0$.
\end{itemize}
\end{corollary}

\begin{proof} (i) $\Leftrightarrow$ (ii),  (i) $\Leftrightarrow$ (v) and (i) $\Leftrightarrow$ (vii) are clear by Theorem
\ref{3.4}.

(i) $\Leftrightarrow$ (iii) is clear by Theorem \ref{3.4}. It is enough to put $B:=C$.

(iv) $\Rightarrow$ (i) is immediate from the implication (iv) $\Rightarrow$ (i) in Theorem \ref{3.4} by putting $B:=C$.

(vi) $\Rightarrow$ (i) is clear by Theorem \ref{3.4} by putting $B:=R$.

(viii) $\Rightarrow$ (i) is clear by Theorem \ref{3.4} by putting $B:=C$ and $C:=R$.

Finally, the implications (iii) $\Rightarrow$ (iv), (v) $\Rightarrow$ (vi) and  (vii) $\Rightarrow$ (viii) are obvious.
\end{proof}

\section{Cotorsion Theories}

We begin this section by recalling the definition of cotorsion theories.

Let $\mathscr{F}$ be a class of $R$-modules. Let $\mathscr{F}^{\perp}$ denote the class of all $R$-modules $L$ such that
$\Ext^1_R(F,L)=0$ for all $F\in \mathscr{F}$. Also, let $^{\perp}\mathscr{F}$ denote the class of all $R$-modules $L$ such
that $\Ext^1_R(L,F)=0$ for all $F \in \mathscr{F}$.

\begin{definition}  A pair $(\mathscr{F},\mathscr{C})$ of classes of $R$-modules is said to be a {\it cotorsion theory},
if $\mathscr{F}^{\perp}=\mathscr{C}$ and $^{\perp}{\mathscr{C}}=\mathscr{F}$.
\end{definition}

\begin{theorem}\label{4.3} Let $C$ be a  semidualizing module for $R$. The following are equivalent:
\begin{itemize}
\item[(i)] $C$ is projective.
\item[(ii)] $(\mathscr{F}_C (R),\mathscr{F}_C(R)^{\perp})$ is a cotorsion theory.
\item[(iii)] $\mathscr{F}(R)=\mathscr{F}_C (R)$.
\item[(iv)] $\mathscr{F}(R)\subseteq \mathscr{F}_C (R)$.
\item[(v)] $(\mathscr{P}_C (R), \mathscr{P} _C (R)^{\perp})$ is a cotorsion theory.
\item[(vi)] $\mathscr{P}(R)=\mathscr{P} _C (R)$.
\item[(vii)] $\mathscr{P} (R)\subseteq \mathscr{P} _C (R)$.
\item[(viii)] $(^{\perp} \mathscr{I} _C (R), {\mathscr{I}} _C(R))$ is a cotorsion theory.
\item[(ix)] $\mathscr{I} (R)= \mathscr{I} _C (R)$.
\item[(x)] $\mathscr{I} (R)\subseteq \mathscr{I} _C (R)$.
\end{itemize} Furthermore, if $R$ is local, then the above conditions are equivalent to the following:
\begin{itemize}
\item[(xi)] $C\cong R$.
\end{itemize}
\end{theorem}

\begin{proof} The proof for the projective parts is similar to the flat parts, and so we only prove the flat parts.

(i) $\Rightarrow$ (ii) is clear by Lemma \ref{3.2} and \cite[Lemma 3.4.1]{10}. Note that, by Lemma \ref{3.1}, $C$
is projective if and only if $[C]\approx [R]$.

(ii) $\Rightarrow$ (i) We have $R \in  ^{\perp}(\mathscr{F}_C(R)^{\perp})=\mathscr{F}_C(R)$. Hence, $R\cong C\otimes_RF$
for some flat $R$-module $F$, and so $[C]\approx [R]$ by the equivalence (i) $\Leftrightarrow$ (ii) in \cite[Proposition 4.1.4]{6}.
Now, Lemma \ref{3.1} implies that $C$ is projective.

(i) $\Rightarrow$ (iii) holds by Lemma \ref{3.2}.

(iii) $\Rightarrow$ (iv) is obvious.

(iv) $\Rightarrow$ (i) This can be shown by the same proof as the proof of (ii) $\Rightarrow$ (i).

(i) $\Rightarrow$ (viii) follows by Lemma \ref{3.2}, because clearly $(^{\perp}\mathscr{I}(R),{\mathscr{I}}(R))$
is a cotorsion theory.

(viii) $\Rightarrow$ (i) We have $\mathscr{I} (R)\subseteq (^{\perp} {\mathscr{I}}_C (R))^{\perp}=\mathscr{I}_C (R)$.
Let $E$ be a faithfully injective $R$-module. Then $E\in \mathscr{I}_C(R)$, and so $C$ is projective by Lemma \ref{3.1}.

(i) $\Rightarrow$ (ix) is clear by Lemma \ref{3.2}.

(ix) $\Rightarrow$ (x) is obvious.

(x) $\Rightarrow$ (i) This can be shown by the same proof as the proof of (viii) $\Rightarrow$ (i).

Now, assume that $R$ is local. Then, it is known that $C$ is projective if and only if $C\cong R$.
\end{proof}

\begin{definition} Let $C$ be a semidualizing module for $R$ and $n$ a non-negative integer.
\begin{itemize}
\item[(i)] $\mathscr{P}^n(R):=\{M\mid \pd_R M\leq n \}$.
\item[(ii)] $\mathscr{P}_C^n(R):=\{M\mid \mathscr{P} _C\text{-}\pd_R(M)\leq n \}$.
\end{itemize}
\end{definition}

Now, we present our last result.

\begin{theorem}\label{4.6} Let $C$ be a  semidualizing module for $R$ and $n$ a natural number.
The following are equivalent:
\begin{itemize}
\item[(i)] $C$ is projective.
\item[(ii)] $(\mathscr{P}_C^n(R),\mathscr{P}_C^n(R)^{\perp})$ is a cotorsion theory.
\item[(iii)] $\mathscr{P}^n(R)= \mathscr{P}_C^n(R)$.
\item[(iv)] $\mathscr{P}^n(R)\subseteq \mathscr{P}_C^n(R)$.
\end{itemize}
Furthermore, if $R$ is local, then the above conditions are equivalent to the following:
\begin{itemize}
\item[(v)] $C\cong R$.
\end{itemize}
\end{theorem}

\begin{proof} (i) $\Rightarrow$ (ii) is clear by Lemma \ref{3.2} and \cite[Theorem 4.2]{1}.

(ii) $\Rightarrow$ (i) We have, $R \in ^{\perp}(\mathscr{P}_C^n(R)^{\perp})=\mathscr{P}_C^n(R)$. Hence
$\mathscr{P}_C\text{-}\pd_R(R)\leq n$, and so $C$ is projective by Lemma \ref{3.1}.

(i) $\Rightarrow$ (iii) is clear by Lemma \ref{3.2}.

(iii) $\Rightarrow$ (iv) is obvious

(iv) $\Rightarrow$ (i) This can be shown by the same proof as the proof of (ii) $\Rightarrow$ (i).

The last assertion is obvious.
\end{proof}


\end{document}